\documentclass[11pt]{article}
\usepackage[colorlinks]{hyperref}
\usepackage{color}
\usepackage{graphicx}
\usepackage{graphics}
\usepackage{makeidx}
\usepackage{showidx}
\usepackage{latexsym}
\usepackage{amssymb}
\usepackage{verbatim}
\usepackage{amsmath}
\usepackage{amsthm}
\usepackage{amsfonts}
\usepackage{subcaption} 
\usepackage[abs]{overpic}
\usepackage{xcolor,varwidth}
\usepackage[title]{appendix}
\usepackage{leftidx}
\usepackage{blindtext}
\usepackage{titlesec}
\usepackage[mathscr]{euscript}

\usepackage{geometry}
\newtheorem{theorem}{Theorem}[section]
\newtheorem{proposition}[theorem]{Proposition}
\newtheorem{remark}[theorem]{Remark}

\newtheorem{question}[theorem]{Question}
\newtheorem{definition}[theorem]{Definition}
\newtheorem{corollary}[theorem]{Corollary}
\newtheorem{lemma}[theorem]{Lemma}
\newtheorem{example}[theorem]{Example}

\title{Anosov contact metrics, Dirichlet optimization and entropy}

\author{Surena Hozoori}
\newcommand{\Addresses}{{
  \bigskip
  \footnotesize

Surena Hozoori, \textsc{Department of Mathematics, Brandeis University.}\par\nopagebreak
  \textit{E-mail address}: \texttt{hozoori@brandeis.edu}
  
  }}
    \date{}
\begin{document}
\maketitle
\noindent
\begin{abstract}
The first main result of this paper classifies contact 3-manifolds admitting critical metrics, i.e. adapted metrics which are the critical points of the Dirichlet energy functional. This gives a complete answer to a question raised by Chern-Hamilton in 1984. Secondly, we show that in the case of Anosov contact metrics, the optimization of such energy functional is closely related to Reeb dynamics and can be described in terms of its entropy. We also study the consequences in the curvature realization problem for such contact metrics. 
\end{abstract}


\section{Introduction}

In geometric analysis, it is classical to study optimization problems for functionals associated with a geometric structure on a manifold. A famous example is the celebrated {\em Yamabe problem}, resolved affirmatively in 1984, which asks whether a smooth Riemannian metric on a closed manifold is conformally equivalent to one with constant scalar curvature. See \cite{yam,yam1,yam2}. An intermediate step in the study of such problem is to understand the scalar curvature functional defined on the space of Riemannian metrics by appropriately integrating such geometric quantity over the underlying manifold. In particular, applying standard {\em variational techniques}, one would hope to achieve a geometric interpretation of the critical condition for such functional, and in this example, it can be shown that, when restricted to Riemannian metrics of constant volume, a Riemannian metric is critical for the scalar curvature functional, if and only if, it is {\em Einstein}.

Motivated by similar trends in Riemannian geometry, Chern and Hamilton initiate the study of various energy functionals for {\em contact manifolds} in their 1984 seminal paper \cite{ch}. In particular, to explore the interactions between $CR$ and contact geometries, they define the {\em Dirichlet energy} of an almost complex structure $\phi$ adapted to a given contact manifold $(M,\alpha)$ as
$$\mathcal{E}(\phi):=\int_M ||\mathcal{L}_{X_\alpha} \phi||^2\  \alpha \wedge d\alpha,$$
where $X_\alpha$ is the associated Reeb flow,
and ask when such energy functional attains its minimum for some almost complex structure $\phi$. One can think of such minimizer as the {\em nicest} almost complex structure adapted to a contact manifold. In particular, they conjecture that the minimizer of Dirichlet energy functional exists when $X_\alpha$ is {\em regular}, i.e. the orbits of $X_\alpha$ are closed and their orbit space is Hausdorff.

We note that this problem can be reformulated both in terms of CR geometry and {\em contact Riemannian} geometry. More specifically, choosing such almost complex structure $\phi$ gives rise to a CR structure whose associated {\em Tanaka connection} $\leftidx{^*}{\nabla}$ might have non-trivial torsion tensor $\leftidx{^*}{T}$, where vanishing torsion corresponds to the integrability of such CR structure and hence, the existence of a {\em Sasakian structure}. As observed in \cite{bl,tanno}, we have
$$2||\leftidx{^*}{T} ||^2=||\mathcal{L}_{X_\alpha} \phi||^2+8(n-1),$$
where $M$ is $(2n+1)$-dimensional. Therefore, optimization of the above Dirichlet energy is equivalent to the optimization problem for the Tanaka torsion. 

In this manuscript, we adopt a Riemannian geometric approach as it is natural to interpret our computations in terms of {\em curvature} as well. More specifically, the almost complex structure $\phi$ together with the area form $d\alpha|_{\xi}$ give rise to a Riemannian tensor on the contact structure $\xi:=\ker{\alpha}$, which then can be extended to $TM$ by letting $X_\alpha$ be orthonormal to $\xi$. Equivalently, we call $g$ a {\em contact metric} adapted to $(M,\alpha)$ if
$$d\alpha=2*\alpha,$$
where $*$ is the Hodge star associated with $g$ (See Definition~\ref{cmdef} and the subsequent remarks).

Since $\alpha$ is invariant under $X_\alpha$, it is easy to see that for such contact metric, we have $||\mathcal{L}_{X_\alpha} \phi||=||\mathcal{L}_{X_\alpha} g||$ and therefore, the following reformulation of the Chern-Hamilton question in terms of contact metrics is in order. See Chapter~10 of \cite{bl} for a thorough reference on this problem and the related history.

\begin{question}[Chern-Hamilton question in terms of contact metrics \cite{ch}]\label{ch}
Suppose $(M,\alpha)$ is a contact manifold and $\mathcal{M}(\alpha)$ is the associated space of adapted contact metrics. When does there exist a contact metric $g\in\mathcal{M}(\alpha)$ for which the minimum of the Dirichlet energy functional \\
$\mathcal{E}:\mathcal{M}(\alpha)\rightarrow \mathbb{R}$ defined by
$$\mathcal{E}(g):=\int_M ||\mathcal{L}_{X_\alpha} g||^2\  \alpha \wedge d\alpha$$
is achieved?
\end{question}

The tensor $\tau:=\mathcal{L}_{X_\alpha} g$ is naturally called the {\em torsion} of a contact metric in the literature. Moreover, it can be seen that for a contact metric $g$, we have $$Ricci(X_\alpha)=2n-2||\mathcal{L}_{X_\alpha} \phi||^2,$$ where $Ricci(X_\alpha)$ is the Ricci curvature associated with the Reeb direction, i.e. \\$Ricci(X_\alpha):=\kappa(e,X_\alpha)+\kappa(\phi e,X_\alpha)$ is the sum of the {\em $\alpha$-sectional curvatures}. Therefore, this question can be interpreted as an optimization problem for the integral of such Ricci curvature as well. As a matter of fact, the Chern-Hamilton question in terms of such Ricci curvature functional was asked few years earlier by Blair \cite{blc}.

\vskip0.5cm

The primary purpose of this paper is to study the Chern-Hamilton question in dimension 3, giving a complete answer to the Chern-Hamilton question by classifying all {\em critical contact metrics} for the Dirichlet energy functional.  As we will see, such classification relies on interactions with {\em Anosov dynamics}. Therefore, our analysis of the situation extends to understanding the infimum of such functional for {\em Anosov contact metrics}, revealing deeper interplay with {\em Reeb dynamics} and introducing new relevant questions.

\vskip0.5cm

\noindent\fbox{%
    \parbox{\textwidth}{%
\textbf{Assumptions:} In this paper, unless stated otherwise, $M$ is a closed, oriented connected 3-manifold, $\alpha$ is a smooth, i.e. $C^\infty$, positive contact form on $M$ and $X_\alpha$ is the Reeb vector field associated with $\alpha$. Moreover, the flow generated by a vector field $X$ is denoted by $X^t$.        
    }%
}
\vskip0.5cm

An important ingredient in our study is {\em Tanno's varational principle} for critical metrics of the Dirichlet energy functional introduced in 1989 \cite{tanno}. He proves that the critical condition is in fact equivalent to the geometric condition:
$$\nabla_X (\mathcal{L}_{X_\alpha} g)=2(\mathcal{L}_{X_\alpha} g) \ \phi \ \ \ \ \ \ \ \ \ \ \ \ \text{(Tanno's variational principle)}$$

Subsequently, Deng \cite{deng} improves Tanno's result in 1991 by showing that all such critical contact metrics are in fact the minimizers of the Dirichlet energy functional. 
The next progress occurred in 1995, when Rukimbira \cite{ruk} classified all contact metrics with vanishing torsion, i.e. $\tau \equiv 0$, a condition which is usually referred to as {\em K-contactness} and in dimension 3, is known to be equivalent to such contact metric being {\em Sasakian}. More specifically, Rukimbira shows that a contact manifold $(M,\alpha)$ satisfying such condition is virtually equivalent, possibly after a perturbation, to a {\em Boothby-Wang fibration}. These are contact manifolds where $X_\alpha$ traces a $\mathbb{S}^1$-fibration of $M$ (see Section~\ref{s4}). Note that such contact metrics are trivially the minimizers of the Dirichlet functional as they satisfy $\mathcal{E}(g)=0$. Therefore, it is not required to appeal to Tanno's variational principle to answer the Chern-Hamilton question for such class of contact manifolds. This in fact proves the Chern-Hamilton conjecture in the case of {\em (almost) regular} contact manifolds, leaving the question wide open if one drops the regularity assumption.

In 2005, Perrone \cite{torsion} studied critical contact metrics with nowhere vanishing torsion exploiting the Tanno's equation. In particular, he showed that in such cases, the Reeb vector field $X_\alpha$ satisfies some {\em weak hyperbolicity} condition referred to as {\em projective (or conformal) Anosovity}. His approach is instrumental to our study, as his work is the first at hinting towards the interplay between the Dirichlet energy functional and Anosov dynamics. Also see \cite{taut,abb,hoz1} for recent progress on the Chern-Hamilton question.

We improve Perrone's observation by adding dynamical arguments which eventually yield a classification of critical contact metrics, giving a complete answer to the Chern-Hamilton question in dimension 3. More specifically, we first show that for a critical contact metric, the {\em scalar torsion} $||\tau||$ is a constant function on $M$. Then, the case $||\tau||\equiv 0$ is classified by the mentioned work of Rukimbira. The case $||\tau||\equiv C>0$ is also resolved by appealing to a classical rigidity results in Anosov dynamics \cite{green,entropy}. In fact, we are able to show that in this case, the critical contact metric has close relation to $\widetilde{SL}(2,\mathbb{R})$ geometry, while from a dynamical viewpoint, this means that $X_\alpha$ is virtually $C^\infty$-conjugate to the geodesic flow of a surface of constant negative curvature. We call such contact manifolds {\em algebraic Anosov contact manifolds} (see Section~\ref{s4}).

\begin{theorem}\label{introclass}(Theorem~\ref{class})
Let $(M,\alpha)$ be a contact 3-manifold such that the Dirichlet energy functional admits a minimizer. Then,

(i) $(M,\alpha)$ can be approximated by generalized Boothby-Wang fibrations (the case $||\tau||\equiv 0$);

\noindent or

(ii) $(M,\alpha)$ is smoothly strictly contactomorphic to an algebraic Anosov contact manifold (the case $||\tau||\equiv C>0$).

\noindent Conversely, for the contact manifolds listed above the Dirichlet functional admits a minimizer.
\end{theorem}

We remark that a {\em generalized Chern-Hamilton question} has been asked \cite{hoz1} by allowing the contact form $\alpha$ vary, while only fixing $\xi=\ker{\alpha}$. The above theorem answers both versions of the Chern-Hamilton question, as the minimizer of such {\em generalized Dirichlet functional} $\widetilde{\mathcal{E}}(g)$ would be in particular the minimizer of the one defined by Chern-Hamilton for the contact form $\alpha$ associated with the minimizer of $\widetilde{\mathcal{E}}(g)$.

\begin{remark}
A theorem equivalent toTheorem~\ref{introclass} was simultaneously and independently proved by Mitsumatsu-Peralta-Salas-Slobodeanu \cite{mps} and exploiting bi-contact geometry.
\end{remark}

From the viewpoint of contact geometry, the above theorem means that the minimizer of the Dirichlet energy functional rarely exists, as the contact manifolds included in the above theorem constitute a very limited list. In particular, all these contact manifolds are {\em tight} and {\em symplectically fillable} by classical results in contact topology \cite{lisca,mcduff}.

However, the dynamical interpretation is more delicate. It is well known in Anosov dynamics that the {\em algebraic Anosov contact flows}, which appear in the above theorem as the Reeb flows associated with algebraic Anosov contact manifolds, are unique at satisfying many rigidity properties among {\em Anosov contact flows}. More specifically, Ghys \cite{ghys} shows that they are the only Anosov contact flows (up to reparametrization) for which the weak invariant bundles are $C^2$ (regularity rigidity theorem) and Foulon \cite{entropy} shows that they are the only examples where the topological and measure entropies agree (entropy rigidity theorem). Also, from a topological point of view, algebraic Anosov contact flows are, up to orbit equivalence, the only Anosov contact flows on Seifert fibered manifolds \cite{ghys2}. As a matter of fact, these were the only known examples of Anosov contact flows for few decades. But thanks to the celebrated result of Foulon-Hasselblatt \cite{foulonh}, we now know that Anosov contact flows exist in much more abundance, including on many hyperbolic manifolds. Therefore, in the shadow of the beautiful theory of Anosov dynamics, our result can be interpreted as a new rigidity result for algebraic Anosov contact flows in terms of contact metric geometry. These are Anosov contact flows admitting a critical contact metric.

To further elaborate on the dynamical perspective, we study the contact metrics associated with a general (possibly non-algebraic) Anosov contact manifold, using the classical ergodic theory of Anosov systems, revealing that the appearance of the entropy rigidity result of Foulon in the above classification theorem is not a coincidence. In fact, even in the non-algebraic case, the infimum of the Dirichlet energy functional can be computed in terms of the measure entropy of the associated Reeb flow.

\begin{theorem}\label{introinf}(Theorem~\ref{inf})
Let $(M,\alpha)$ be an Anosov contact 3-manifold. Then,
$$\inf_{g \in \mathcal{M}(\alpha) } \mathcal{E}(g)=\frac{\mathtt{h}^2_{\alpha \wedge d\alpha}(X_\alpha)}{Vol(\alpha \wedge d\alpha)},$$
where $\mathtt{h}_{\alpha \wedge d\alpha}(X_\alpha)$ is the measure entropy of the invariant measure $\alpha \wedge d\alpha$ under the Reeb flow generated by $X_\alpha$. Such infimum is achieved exactly when $(M,\alpha)$ is an algebraic Anosov contact manifold.
\end{theorem}

In other words, at least for an Anosov contact manifold, the Dirichlet energy functional partly measures the chaotic behavior of the associated Reeb flow, and partly measures the {\em geometric distortion} with respect to an algebraic $\widetilde{SL}(2,\mathbb{R})$-model, a distortion which is unavoidable except in the case of algebraic Anosov contact manifolds. Note that the conclusion of Theorem~\ref{introinf} also holds in the extreme opposite case of Anosov contact manifolds, i.e. Boothby-Wang fibrations where such entropy vanishes. Therefore, we ask whether the same formula holds for a general contact manifold.
\begin{question}\label{conjinf}
Does the Liouville entropy formula
$$\inf_{g \in \mathcal{M}(\alpha) } \mathcal{E}(g)=\frac{\mathtt{h}^2_{\alpha \wedge d\alpha}(X_\alpha)}{Vol(\alpha \wedge d\alpha)}$$
hold for an arbitrary contact 3-manifold $(M,\alpha)$?
\end{question}

We note that if the above conjecture holds, a recent result of \cite{abbreeb} implies that the infimum of the generalized Dirichlet energy $\widetilde{\mathcal{E}}(g)$ (i.e. letting the contact form $\alpha$ vary as well) equals zero, as they show that for any contact structure, there is a Reeb flow with arbitrary small entropy.

As previously mentioned, our computations of the Dirichlet energy functional in the above theorems can be interpreted in terms of the curvature of contact metric, yielding implications on classical curvature realization problems for contact metrics studied in the literature, e.g. \cite{bl,ol,hoz2}. One immediate corollary of Theorem~\ref{introinf} is the following.

\begin{corollary}(Corollary~\ref{ricciglobal})
If $(M,\alpha)$ is an Anosov contact manifold, we have
$$\int_M Ricci(X_\alpha) \ \alpha \wedge d\alpha \leq 2- \frac{2\mathtt{h}^2_{\alpha \wedge d\alpha}(X_\alpha)}{Vol(\alpha \wedge d\alpha)},$$
where the equality can be obtained only in the case of algebraic Anosov contact manifolds.
\end{corollary}

To the best of our knowledge, this is the first such result imposing a global obstruction on functions which can be realized as $Ricci(X_\alpha)$. We note that for an arbitrary contact manifold, we can always define a contact metric with vanishing torsion (or equivalently, with $Ricci(X_\alpha)=2$) in any small enough neighborhood. Therefore, the above corollary gives a truly global obstruction.

Not surprisingly, we can also study asymptotic behavior of curvature quantities as we approach an {\em ideal} $\widetilde{SL}(2,\mathbb{R})$-model. This can be seen as a refinement of the convergence result embodied in Theorem~\ref{introinf}.

\begin{theorem}(Theorem~\ref{pinch})(Asymptotic curvature pinching for Anosov contact metrics)
Suppose $(M,\alpha)$ is an Anosov contact 3-manifold, $\bar{\mathtt{h}}:=\frac{\mathtt{h}_{\alpha \wedge d\alpha}(X_\alpha)}{Vol(\alpha \wedge d\alpha)}$ is the Liouville entropy of such flow, $V=Vol(\alpha \wedge d\alpha)$ and $\epsilon>0$. There exists a sequence of contact metrics $\{g_i \}_{i\in \mathbb{N}}$ adapted to $(M,\alpha)$, such that their Ricci and sectional curvature operators $Ricci_i$ and $\kappa_i$ satisfy the following:

(1) the sequence of smooth functions $\{Ricci_i(X_\alpha)\}_{i\in\mathbb{N}}$ converge in Lebegue measure to the constant function $R\equiv 2-2\bar{\mathtt{h}}^2$,

(2) both sequences of smooth functions $\{ \kappa_i(E^s)\}_{i\in\mathbb{N}}$ and $\{ \kappa_i(E^u)\}_{i\in\mathbb{N}}$ converge in Lebegue measure to the constant function $\kappa \equiv 1-\bar{\mathtt{h}}^2$, where $\kappa_i(E^u)$ and $\kappa_i(E^s)$ are the sectional curvature functions corresponding to the invariant bundles $E^u$ and $E^s$, respectively,

(3) if $\gamma$ is a periodic orbit of $X_\alpha$ with period $T$ and the eigenvalues of its return map corresponding to $E^u$ and $E^s$ being $\lambda_u$ and $\lambda_s$, respectively, then $\{Ricci_i(X_\alpha)\}_{i\in\mathbb{N}}|_\gamma$ converges uniformly to $2-2(\frac{\ln{|\lambda_u}|}{T})^2=2-2(\frac{\ln{|\lambda_s|}}{T})^2$ and similarly, both $\kappa_i(E^u)|_\gamma$ and $\kappa_i(E^s)|_\gamma$ converge uniformly to $1-(\frac{\ln{|\lambda_u}|}{T})^2=1-(\frac{\ln{|\lambda_s|}}{T})^2$.
\end{theorem}

As we observe in the above theorem, even though our computations yield characterizations of some curvature quantities, it is not easy to exploit such formulas in the curvature realization problems for a general Anosov contact manifold, as the functions involved depend on the invariant bundles and their regularity (one can think of this as a general Anosov contact manifold lacking an {\em algebraic structure}). Therefore, we only achieve an asymptotic realization which is in general only as good as the convergence given in the Birkhoff ergodic theorem. However, in the case of algebraic Anosov contact manifolds such difficulty can be bypassed thanks to the rigidity properties of the underlying Anosov contact flow. Therefore, we can obtain a complete characterization of all functions which can be realized as $Ricci(X_\alpha)$, answering a classical curvature realization problem in contact metric geometry for this class of contact manifolds (see \cite{bl,hoz2}), which interestingly turns out to depend on the entropy of the underlying Reeb flow. 

\begin{theorem}(Theorem~\ref{alg})(Ricci-Reeb realization formula for algebraic Anosov contact manifolds)
Let $(M,\alpha)$ be an algebraic Anosov contact manifold with Liouville entropy $\bar{\mathtt{h}}=\frac{\mathtt{h}_{\alpha \wedge d\alpha}}{Vol(\alpha\wedge d\alpha)}$. Then, for a smooth real function $f:M\rightarrow \mathbb{R}$, the followings are equivalent:

(1) For some adapted contact metric, we have $Ricci(X_\alpha)=f$ everywhere.

(2) for real functions $\eta,\sigma:M\rightarrow \mathbb{R}$ we have
$$f=2-2(\bar{\mathtt{h}}+X_\alpha\cdot \eta)^2-2[X_\alpha\cdot \sigma-2\sigma(\bar{\mathtt{h}}+X_\alpha\cdot \eta) ]^2.$$
\noindent In particular, if $(UT\Sigma,\alpha)$ is the canonical contact manifold corresponding to a surface of constant curvature $K<0$, a function $f$ can be realized as $Ricci(X_\alpha)$, if and only if, it can be written as
$$f=2-2(\sqrt{-K}+X_\alpha\cdot \eta)^2-2[X_\alpha\cdot \sigma-2\sigma(\sqrt{-K}+X_\alpha\cdot \eta) ]^2,$$
for some functions $\eta,\sigma:M\rightarrow \mathbb{R}$.
\end{theorem}


\begin{example}
A special family of examples worth noticing are algebraic Anosov contact manifolds with constant $Ricci(X_\alpha)$ (i.e. when $f\equiv C$ in Theorem~\ref{alg}). One can easily construct such examples using the existing $\widetilde{SL}(2;\mathbb{R})$ geometry (corresponding to $\sigma\equiv 0$ and constant $\eta$ in the above theorem), as long as it respect the bound enforced by the entropy, i.e. $C\leq 2-2\bar{\mathtt{h}}^2$. In the case of the geodesic flow for a surface of constant curvature $K<0$, this means $C\leq 2+2K$.

\end{example}



It is noteworthy that Perrone shows that having constant scalar torsion, or equivalently, having constant $Ricci(X_\alpha)$, is not limited to critical contact metrics and any {\em homogeneous} contact metric has this property \cite{hom}.

 \vskip1cm
 
 We first discuss the necessary background on contact manifolds and their compatible Riemannian geometry in Section~\ref{s2} and bring the related facts about the geometry and ergodic theory of Anosov systems in Section~\ref{s3}. Section~\ref{s4} will be devoted to the classification of critical contact metrics and giving a complete answer to the Chern-Hamilton question (Theorem~\ref{introclass}) and in Section~\ref{s5}, we explore the relation between Dirichlet energy functional and measure entropy of Reeb flows in the case of Anosov contact manifolds (Theorem~\ref{introinf}). Finally in Section~\ref{s6}, we discuss the implications of our work on the curvature realization problems on Anosov contact manifold. 

\vskip1cm

\textbf{ACKNOWLEDGEMENT:} We are grateful to Daniel Peralta-Salas and Radu Slobodeanu for helpful conversations around the conjecture addressed in this paper. 
We are also thankful to Domenico Perrone and John Etnyre for commenting on an earlier version of this paper and to Thomas Barthelmé for introducing the entropy rigidity result of Patrick Foulon to the author, which played an important role in this paper. Finally, we sincerely appreciate the helpful and thorough feedback by the referees.


\section{Contact metrics and Dirichlet energy}\label{s2}

In this section, we overview the necessary background from contact geometry in dimension 3 and its associated Riemannian geometry. One should consult \cite{geiges} for general facts on contact manifolds and \cite{bl} for the Riemannian geometric approach.

Recall that if $M$ is an oriented 3-manifold, the 1-form $\alpha$ is called a {\em contact form} on $M$ if $\alpha \wedge d\alpha$ is nowhere vanishing. We call such $\alpha$ a {\em positive} contact form if $\alpha \wedge d\alpha>0$ with respect to the given orientation on $M$, and otherwise, a {\em negative} contact form. We call the plane field $\xi:=\ker{\alpha}$ a positive or negative {\em contact structure}, respectively. In this paper, unless stated otherwise, by contact form (structure, manifold) we refer to a positive contact form (structure, manifold). The pair $(M,\alpha)$ is called a {\em contact manifold} in this paper. We call two contact manifolds $(M,\alpha)$ and $(\tilde{M},\tilde{\alpha})$ {\em strictly contactomorphic}, if there exists a diffeomorphism $\psi:M\rightarrow \tilde{M}$ with $f^*\tilde{\alpha}=\alpha$. In this paper, contact forms are smooth, i.e. $C^\infty$, unless stated otherwise.

Associated to a contact manifold $(M,\alpha)$, there is a unique vector field $X_\alpha$, called the {\em Reeb vector field}, satisfying the following conditions
$$\alpha(X_\alpha)=1\ \ \ \ \ \ \ \ \ \ \ \ \text{and} \ \ \ \ \ \ \ \ \ \ \ \iota_{X_\alpha}d\alpha=0.$$

Note that these conditions imply that $X_\alpha$ is transverse to the contact structure $\xi$ and it preserves the contact form, i.e. $\mathcal{L}_{X_\alpha}\alpha=0$. In particular, $X_\alpha$ preserves both the area form $d\alpha|_\xi$ on the underlying contact structure and the volume form $\alpha\wedge d\alpha$ on $TM$.

\subsection{Compatible Riemannian geometry for contact 3-manifolds}\label{s21}

Given a contact manifold $(M,\alpha)$, we can naturally restrict Riemannian geometry to the subclass of Riemannian structures satisfying a natural {\em compatibility condition} with respect to the contact form $\alpha$, namely the {\em contact metrics}. To obtain such Riemannian structure, it suffices to employ an almost complex structure $\phi$ on $\xi=\ker{\alpha}$ (defining the rotation by $\pm\frac{\pi}{2}$). Then, the area form $d\alpha$ can be used to construct a Riemannian tensor on $\xi$ and we can naturally extend this tensor to a Riemannian metric by imposing $X_\alpha$ to be orthonormal to $\xi$.

\begin{definition}\label{cmdef}
Let $(M,\alpha)$ be a 3-manifold equipped with a positive contact form. The Reimannian metric $g$ is called a {\em contact metric}, if
$$d\alpha=2\ast \alpha,$$
where $\ast$ is the Hodge star operation induced from $g$.

Equivalently,  
$g$ is a contact metric adapted to $(M,\alpha)$ with Reeb vector field $X_\alpha$, if
there exist a $(1,1)$ tensor $\phi$ satisfying
$$g(X_\alpha,.)=\alpha(.),$$
$$g(.,\phi .)=\frac{1}{2}d\alpha(.,.),$$
$$\phi^2=-I+\alpha \otimes X_\alpha.$$


\end{definition}

In this paper, we assume contact metrics are smooth, i.e. $C^\infty$. Note that this is equivalent to $\phi$ in the above definition to be smooth, as $\alpha$ is already assumed smooth.

\begin{remark}
We note that in the above definition, the orientation the almost complex structure $\phi$ induces on $\xi=\ker{\alpha}$ is opposite the orientation given by $d\alpha|_\xi$. One can also define a contact metric using an almost complex structure with the other orientation convention, and both such definitions are in fact equivalent, as in the first formulation of contact metrics in terms of the Hodge star operation, one does not need to refer to an almost complex structure $\phi$ at all. Therefore, all the results of this paper are independent of such choice. The literature of contact metric geometry has adopted both conventions regularly. We have chosen our convention to be compatible with the one used in the main papers our work relies on, i.e. \cite{tanno,bl,torsion}, since we would like to employ relevant formulas from the literature avoiding confusion about the signs. If one wants to compare the computations of this paper to ones using the other convention \cite{etnyre,etnyre2,hoz1,hoz2}, they should be cautious of some sign changes in the relevant formulas.
\end{remark}

\begin{remark}
One can extend the above definition by assuming the area induced by the metric on $\xi=\ker{\alpha}$ to be an arbitrary constant multiple of $d\alpha|_\xi$, i.e.
$$d\alpha=r\ast \alpha,$$
for an arbitrary constant $r>0$. Such Riemannian metric is usually referred to as a {\em compatible metric}. See \cite{etnyre,etnyre2,hoz1,hoz2,radu}. It is well known that the choice of such constant does not affect any geometric phenomena intrinsic to the underlying contact structure, as the main geometric duality is about whether $\mathcal{L}_{X_\alpha}\phi$ vanishes at a point (this is discussed in depth in \cite{hoz2}). In particular, all the results of this paper are true independent of such choice of constant (one needs to adjust the curvature formulas appropriately) and our choice of $r=2$ is based on the fact that the majority of the literature on contact metric geometry, specifically the main papers this work is influenced by, follow the same convention. One can further generalize the definition by letting $r$ be a non-constant positive function, usually called {\em weak compatibility}, and it is known that geometry can be fundamentally different for such classes of metrics \cite{etnyre,etnyre2,radu}.
\end{remark}

We define the symmetric tensor $h:=\frac{1}{2}\mathcal{L}_{X_\alpha}\phi$, which plays an important role in this paper. The following basic properties can be found in \cite{bl}.


\begin{proposition}
The following properties hold for a contact metric:

(1) $2dVol(g)=\alpha\wedge d\alpha$.

(2) $\nabla_{X_\alpha} X_\alpha=0$, i.e. $X_\alpha$ is a geodesic field.

(3) $\nabla_{X_\alpha} \alpha=0$ and $\nabla_{X_\alpha} \phi=0$.

(4) $h$ is a symmetric operator, $h\phi=-\phi h$, $Tr(h)=0$ and $h(X_\alpha)=0$.

(5) $\nabla X_\alpha=-\phi-\phi h$.

(6) $\nabla_{X_\alpha}h=\phi-\phi h^2-\phi R(X_\alpha,.)X_\alpha$, where $R$ is the curvature tensor.
\end{proposition}

The above also implies that at any point, there exists an orthonormal eigenvector basis $(e_1,e_2:=\phi e_1)$ for $h$ such that $h(e_1)=\lambda e_1$ and $h(e_2)=-\lambda e_2$, where $\lambda=||h|| \geq 0$ (note that such $(e_1,e_2)$ induces the opposite orientation on $\ker{\alpha}$ compared to $d\alpha$). In particular, we note that in this case, we have $$\lambda^2=\frac{1}{2}Tr(h^2)=-Det(h).$$

We also observe
$$g(\nabla_{e_i}X_\alpha,e_i)=g(-\phi e_i -\phi h(e_i),e_i)=g(-\phi e_i \mp \lambda \phi e_i,e_i)=0.$$

In fact, an alternative viewpoint can be achieved via the second fundamental form of $\xi=\ker{\alpha}$. That is, the bilinear form $II:\xi \times \xi \rightarrow \mathbb{R}$ defined by
$$II(u,v):=\alpha(\nabla_u v)=g(X_\alpha, \nabla_u v)=-g(\nabla_u X_\alpha, v).$$


 Not surprisingly, $h$ and $II$ capture information about the curvature of the underlying contact metric. Let $e\in \xi$ be a unit vector. We first compute
 $$II(e,e)=-g(\nabla_e X_\alpha,e)=g(\phi e+\phi h(e),e)=-g(h(e),\phi e)$$
 and
 $$II(e,\phi e)=-g(\nabla_e X_\alpha,\phi e)=g(\phi e+\phi h(e),\phi e)=1+g(h(e),e).$$
  
  Another useful fact is that as noted in \cite{hoz2}, considering the geodesic field $X_\alpha$, one can observe that the vector field $v(t):=X_*^t (e)$, defined locally along an orbit is the unique Jacobi field satisfying $v(0)=e$ and $X\cdot v(t)=\nabla_{v(t)} X$ and in particular,
$$\frac{\partial}{\partial t} \ln{||X_*^t (e)||}\bigg|_{t=0}=\frac{g(\nabla_e X,e)}{||e||^2}=\frac{-II(e,e)}{||e||^2}.$$
  
Note that 
$$Tr(II)=g(\phi e+\phi h(e),e)+g(-e+ h (e),\phi e)=0,$$
indicating the fact that Reeb flows are volume preserving and the lack of symmetry in $II$ captures the non-integrability of $\xi$:
$$II(e, \phi e)-II(\phi e,e)=g([e,\phi e],X_\alpha)=\alpha([e, \phi e])=-d\alpha(e, \phi e)=2.$$

We record these facts as well as Proposition~3.9 of \cite{hoz2} in the following:

\begin{proposition}\label{sec}
We have

(1) $II(e,e)=-||e||^2\frac{\partial}{\partial t} \ln{||X_*^t (e)||}\bigg|_{t=0}$ for any $e\in\ker{\alpha}$

(2) $Tr(II)=II(e,e)+II(\phi e,\phi e)=0$ for any $e\in\ker{\alpha}$,

(3) $II(e, \phi e)-II(\phi e,e)=2$ for any $e\in\ker{\alpha}$,

(4) $h=0$, if and only if, $II(e,e)=0$ for any $e\in \ker{\alpha}$,

(5) If $h\neq 0$, we have $II(e,e)=0$ exactly when $e$ is an eigenvector of $h$.
\end{proposition}
 
 Therefore, we can write $h$ in any basis $(e,\phi e)$ in terms of $II$ (or vice versa):
 $$h=\begin{bmatrix}
 -1+II(e,\phi e) \hfill \ \ \  \ \ \ \ \ \ \ \  -II(e,e)\\ 
 -II(e,e)  \hfill 1-II(e,\phi e)
 \end{bmatrix}_{(e,\phi e)},$$
 which reduces to the following, if we choose the orthonormal eigenvector basis $(e_1,e_2=\phi e_1)$ as above.
 
 $$h=\begin{bmatrix}
 -1+II(e_1,e_2) \hfill  \ \ \  \ \ \ \ \ \ \ \ 0\\ 
0  \hfill 1-II(e_1,e_2)
 \end{bmatrix}_{(e_1,e_2)}
 =\begin{bmatrix}
\lambda \hfill \ \ \  \ \ 0\\ 
0  \hfill -\lambda
 \end{bmatrix}_{(e_1,e_2)}.$$
 and in particular, we observe
 $$\lambda^2=-Det(h)=(1+II(e_1,e_2))^2=1+II(e_1,e_2)^2+2II(e_1,e_2)$$
 $$=1+II(e_1,e_2)(II(e_2,e_1)-2)+2II(e_1,e_2)=1-Det(II).$$
 It is straight forward to see that such quantity is closely related to $Ricci(X_\alpha)$ via the following:

\begin{theorem}[\cite{bl}, H.19 \cite{hoz2}]\label{curcm}
The Ricci curvature of the Reeb vector field $X_\alpha$ can be computed as
$$Ricci(X_\alpha)=2-Tr(h^2)=2+2Det(h)=2-2\lambda^2=2Det(II).$$
\end{theorem}


\subsection{Torsion, Dirichlet functional and critical contact metrics}\label{s22}

In \cite{ch}, Chern-Hamilton define the torsion of a contact metric as follows.
\begin{definition}
Given a contact metric $g$ on $(M,\alpha)$, the tensor defined as
$$\tau:=\mathcal{L}_{X_\alpha} g$$
is called the {\em torsion} of $g$.
\end{definition}

We note that since $\alpha$ is invariant under $X_\alpha$, we have $||\tau||=2||h||$. We call the quantity $||\tau||$ the {\em scalar torsion} of the contact metric $g$.

\begin{proposition}(\cite{tanno,torsion})
The following properties hold for the torsion function $\tau$:

(1) $\tau=2g(h \phi.,.)$

(2) $\tau \phi=-2g(h.,.)$

(3) $\nabla_{X_\alpha} \tau=2g((\nabla_{X_\alpha}h)\phi.,.)$
\end{proposition}

Chern and Hamilton in \cite{ch} defined the Dirichlet energy as follows.


\begin{definition}
The Dirichlet energy functional $E:\mathcal{M}(\alpha)\rightarrow \mathbb{R}$ is defined as
$$\mathcal{E}(g):=\frac{1}{2}\int_M ||\tau||^2 dVol(g).$$
\end{definition}

We note that
$$\mathcal{E}(g)=\frac{1}{2}\int_M ||\tau||^2 dVol(g)=2\int_M ||h||^2 dVol(g)=\int_M \lambda^2 \alpha \wedge d\alpha$$
$$=Vol(\alpha \wedge d\alpha)-\frac{1}{2}\int_M Ricci(X_\alpha) \ \alpha \wedge d\alpha.$$

In particular, if we let $\tilde{\alpha}=C\alpha$ for some constant $C$, and assuming $g$ and $\tilde{g}$ are contact metrics constructed using $\alpha$ and $\tilde{\alpha}$, respectively, with a fixed almost complex structure $\phi$, we have
$$\mathcal{E}(\tilde{g})=\int_M ||\mathcal{L}_{X_{\tilde{\alpha}}} \phi ||^2\ \tilde{\alpha}\wedge d\tilde{\alpha}=\int_M \frac{1}{C^2} ||\mathcal{L}_{X_\alpha} \phi||^2\ C^2 \alpha \wedge d\alpha=\mathcal{E}(g).$$

Therefore, if needed, one can assume $\alpha \wedge d\alpha$ to be a probability measure without loss of generality, and the measure entropy of the Reeb flow appearing in our analysis, as we will see, can be interpreted as {\em Liouville entropy}.

 The variational method of Tanno \cite{tanno} investigates contact metrics which are critical point of the Dirichlet energy functional and we call them {\em critical contact metrics}.

\begin{theorem}[Tanno 89 \cite{tanno}]\label{tanno}
TFAE:

(1) $g$ is a critical contact metric;

(2) $\nabla_X h=2h\phi$;

(3) $\nabla_X \tau=2\tau \phi$.
\end{theorem}

Deng has shown the following, which reduces the pursuit of the minimizers of Chern-Hamilton energy function to investigating critical contact metrics.

\begin{theorem}[Deng 91 \cite{deng}]\label{dength}
Any contact metric minimizes $E$, if and only if, it is a critical contact metric.
\end{theorem}

An immediate consequence of the above is:

\begin{corollary}
If $g$ is a critical contact metric for $(M,\alpha)$ and $||\tau||=0$ (or equivalently, $||h||=0$) at some point $p\in M$, then we have $||\tau||=0$ (or equivalently, $||h||=0$) along the orbit of $X_\alpha$ containing $p$.
\end{corollary}

Perrone \cite{torsion} investigates the case of critical contact metrics with $||\tau||\neq 0$ and reveal the connections to Anosov dynamics. We review those observations here, as they are crucial for what follows.

We can use the basis $(e_1,e_2=\phi e_1)$ as above and write
$$\nabla_{e_1}X_\alpha =-\phi e_1-\phi h(e_1)=(-1-\lambda)e_2$$
$$\nabla_{e_2}X_\alpha=-\phi e_2-\phi h(e_2)=(1-\lambda)e_1$$
and for some function $\mu$ we have
$$\nabla_{X_\alpha}e_1=-\mu e_2$$
$$\nabla_{X_\alpha}e_2=\mu e_1,$$
since $g(\nabla_{X_\alpha}e_1,X_\alpha)=g(\nabla_{X_\alpha}e_1,e_1)=0.$

In order to use Tanno's equation, we compute
$$(\nabla_{X_\alpha}h)(e_1)=\nabla_{X_\alpha}h(e_1)-h(\nabla_{X_\alpha}e_1)=(X_\alpha \cdot \lambda)e_1-2\lambda \mu e_2.$$

Setting this equal to $2h\phi e_1=-2\lambda e_2$ implies $X_\alpha \cdot \lambda=0$ (this generalizes the above corollary) and $\mu=1$. Note that this also holds when $\lambda=0$.

Moreover, the connection to Anosov flows comes from the fact that at the points where $\lambda \neq 0$, we have
$$g([X_\alpha,e_1],e_2)=g(\nabla_{X_\alpha}e_1-\nabla_{e_1}X_\alpha,e_2)=g(-e_2+(1+\lambda)e_2,e_2)=\lambda>0$$
and
$$g([X_\alpha,e_2],-e_1)=g(\nabla_{X_\alpha}e_2-\nabla_{e_2}X_\alpha,-e_1)=g(e_1+(-1+\lambda)e_1,-e_1)=-\lambda<0,$$
implying that the plane fields $\langle X_\alpha,e_1\rangle$ and $\langle X_\alpha,e_2\rangle$ are positive and negative contact structures respectively. The condition of a vector field, here $X_\alpha$, lying in the intersection of a pair of transverse negative and positive contact structures is known to be equivalent to a weak form of {\em hyperbolicty} in dynamics. When satisfied globally, such vector field or flow is called {\em projectively Anosov} \cite{conf,mitsumatsu}. In fact, it is known (see \cite{ol,conf,hoz4}) that when the flow preserves a volume form such condition implies hyperbolicity (and when satisfied globally, implies {\em Anosovity} of the flow). We will later see this explicitly in Lemma~\ref{cpa}. In the next section, we will explore the geometric implications of (projective) Anosovity. But first, we record these observations in the following proposition. Note that these are direct consequences of Tanno's formula (Theorem~\ref{tanno}).

\begin{proposition}[Perrone 05 \cite{torsion}]\label{per}
With the notation above,

(1) $X_\alpha \cdot ||\tau||=X_\alpha \cdot ||h||=X_\alpha \cdot \lambda=0$,

(2) $\nabla_{X_\alpha}e_1=-e_2$ and $\nabla_{X_\alpha}e_2=e_1$.

(3) $[X_\alpha,e_1]=\lambda e_2$ and $[X_\alpha,e_2]=\lambda e_1$. In particular, the plane fields $\langle X_\alpha,e_1\rangle$ and $\langle X_\alpha,e_2\rangle$ are positive and negative contact structures, respectively.
\end{proposition}


\section{Elements from geometry and ergodic theory of \\
Anosov contact 3-flows}\label{s3}

In this section, we review the required background on Anosov 3-flows. A thorough treatment of the topics included in this section can be found in \cite{hyp}.

\begin{definition}\label{anosov}
We call a flow $X^t$ {\em Anosov}, if there exists a splitting $TM=E^{ss} \oplus E^{uu} \oplus \langle X \rangle$, such that the splitting is continuous and invariant under $X_*^t$ and 
$$ ||X_*^t (v)  || \geq e^{Ct}||v || \text{ for any }v \in E^{uu},$$
$$||X_*^t (u) || \leq e^{-Ct}||u ||\text{ for any }u \in E^{ss},$$
where $C$ is a positive constant, and $||.||$ is induced from some Riemannian metric on $TM$. We call the line bundles $E^{uu}$ and $E^{ss}$, the strong unstable and stable directions (or bundles), respectively, and the plane fields $E^u:=E^{uu}\oplus \langle X\rangle$ and $E^s:=E^{ss}\oplus \langle X\rangle$ are referred to as {\em weak} unstable and stable bundles, respectively. Moreover, we call the vector field $X$, the generator of such flow, an {\em Anosov vector field}.
\end{definition}

It is known (but it's not trivial) that the weak invariant plane fields $E^s$ and $E^u$ are $C^1$ \cite{hps} and tangent to a pair of foliations.

In this paper, we are interested in the case where an Anosov flows is {\em contact}. That means the transverse invariant plane field $E^{ss}\oplus E^{uu}$ is a contact structure. This is equivalent to the generator of the flow being the Reeb vector field $X_\alpha$, for some appropriate contact form $\alpha$. For such $\alpha$ we necessarily have $\ker{\alpha}=E^{ss}\oplus E^{uu}$. The geodesic flow of a hyperbolic surfaces on $UT\Sigma$ gives the standard example of this situation (see Section~\ref{s4}). Note that for Anosov contact flows, we have $E^{ss}=\ker{\alpha}\cap E^s$ and $E^{uu}=\ker{\alpha}\cap E^u$ are $C^1$ line bundles.


Few times in this paper, we use a natural generalization of Anosovity condition, which we bring next.

\begin{definition}\label{ca}
We call a flow $X^t$ {\em projectively Anosov}, if its induced flow on $TM/\langle X\rangle$ admits a {\em dominated splitting}. That is, there exists a splitting $TM/\langle X \rangle=E^s \oplus E^u$, such that the splitting is continuous and invariant under $\tilde{X}_*^t$ and 
$$ ||\tilde{X}_*^t (v)  || / ||\tilde{X}_*^t (u) || \geq e^{Ct}||v || /||u ||,$$
for any $v \in E^u$ and $u \in E^s$, where  $C$ is a positive constant, $||.||$ is induced from some Riemannian metric on $TM / \langle X\rangle$ and $\tilde{X}_*^t$ is the flow induced on $TM/\langle X \rangle$.

Moreover, we call the generating vector field $X$, a {\em projectively Anosov vector field}.
\end{definition}


In \cite{mitsumatsu,conf}, it is shown:

\begin{proposition}\label{cabi}
Let $X$ be a nonsingular vector field on $M$. Then, $X$ is projectively Anosov, if and only if, for a pair of transverse negative and positive contact structures $(\xi_-,\xi_+)$, we have $X\subset \xi_+ \cap \xi_-$.
\end{proposition}

An important geometric quantity in this context is the {\em expansion rate} of an invariant bundle, which measure how fast the norm involved in the definition of (projective) Anosovity {\em stretches or contracts} infinitesimally, in the direction of the unstable or stable bundles, respectively.

\begin{definition}
The following real functions $r_u$ and $r_s$ are called {\em the expansion rates }of the unstable and stable bundles, respectively.
$$r_u:=\frac{\partial}{\partial t} \ln{||\tilde{X}_*^t (v)||}\bigg|_{t=0} \ \ \  \ \ \ \ \ \ \ \ r_s:=\frac{\partial}{\partial t} \ln{||\tilde{X}_*^t (u)||}\bigg|_{t=0},$$
where $v\neq 0 \in E^{uu}$ and $u\neq 0 \in E^{ss}$ are arbitrary.
\end{definition}

One should note that the above definition does not depend on the choice of $v$ and $u$ and in general, these expansion rates are only differentiable along the flow (even assuming the norm is induced from some smooth Riemannian metric). However, one can observe that for an appropriate choice of Riemmanian metric on $M$, the expansion rates are $C^1$ (see \cite{simic}, Theorem~2.1). It is useful to characterize the expansion rates in terms of differential forms as well. Equivalent to a choice of norm on $E^{uu}$ (with respect to which we can define $r_u$) is defining the 1-form $\alpha_u$ which vanishes on $E^s$ and $|\alpha_u(e_u)|=1$ for the unit vector $e_u$. Similarly, we can define $\alpha_s$ and the expansion rates are characterized by

$$\mathcal{L}_X \alpha_u=r_u\alpha_u \ \ \  \ \ \ \ \ \ \ \ \mathcal{L}_X\alpha_s=r_s\alpha_s.$$

The following elementary facts can be found in \cite{hoz3}.
\begin{proposition}\label{anosovprop}
For any projectively Anosov flow generated by $X$,

(1) if $e_u$ is a differentiable lift of a (locally defined) unit vector field $\tilde{e}_u$, then $r_ue_u+q_uX=[e_u,X]$ for some function $q_u$. If such lift also has a differentiable invariant linear span (like in the case of Anosov contact flows), we have $r_ue_u=[e_u,X]$.  Similar statements hold for $E^s$.

(2) for $v\in E^u$, we have $||X^t_*(v)||=||v||e^{\int_0^t r_u(\tau)\ d\tau}$. In particular, if $\gamma$ is a periodic orbit of $X$ with period $T$, the eigenvalue of the first return map of $X$ along $\gamma$ corresponding to $E^u$ has magnitude $e^{\int_0^T r_u(\tau)\ d\tau}$, where $r_u(\tau):=r_u\circ X^\tau$. Similar statements hold for $E^s$.

(3) if $\mathcal{L}_{X_\alpha} \alpha_u=r_u\alpha_u$ and $\tilde{\alpha}_u=e^f \alpha_u$ for some function $f$, then the expansion rate corresponding to $\tilde{\alpha}_u$, i.e. the function $\tilde{r}_u$ satisfying $\mathcal{L}_{X_\alpha} \tilde{\alpha}_u=\tilde{r}_u\tilde{\alpha}_u$, can be computed as $\tilde{r}_u=r_u+X\cdot f$. Similar statements hold for $\alpha_s$.
\end{proposition}

It is easy to see that for a projectively Anosov flow, we have $r_u>r_s$ and Anosovity is equivalent to the existence of a metric with respect to which $r_u>0>r_s$.

Now in the presence of a contact metric $g$ for an Anosov contact form, applying Proposition~\ref{sec}, we have

\begin{proposition}\label{cex}
For any $v\in E^{uu}$ and $v\in E^{ss}$, we have
$$II(v,v)=-g(\nabla_{v} X,v)=-||v||^2r_u$$
and
$$II(u,u)=-g(\nabla_{u} X,u)=-||u||^2r_s.$$
\end{proposition}

These expansion rates capture the local geometry of a Anosov contact flow. But using standard ergodic theory, such interaction transcends to a global one involving {\em entropy} of the Reeb flow, considering the fact that Anosov contact flows are mixing and ergodic with respect to the smooth invariant volume form $\alpha \wedge d\alpha$. Recall that the {\em measure entropy} is a quantity associated with a dynamical system which describes the amount of chaotic behavior with respect to a measure. See Appendix~A of \cite{hyp} for a formal introduction and basic properties.

Mainly, we are interested in {\em Pesin's entropy formula}, which states that the integral of the unstable expansion rate with respect to a smooth invariant probability measure yields its measure entropy. Such invariant probability measure is sometimes called the {\em Liouville measure} and its entropy is called the {\em Liouville entropy}. Of course, we can use $\alpha \wedge d\alpha/||\alpha \wedge d\alpha||$ as the Liouville measure on $M$. See Chapter~7 of \cite{hyp} for thorough discussion on this and more on measure entropy (note that their definition of expansion rate and the formula below differs from ours in a sign).

\begin{theorem}[Pesin's entropy formula] 
Assume $X$ generates a mixing (in particular, contact) Anosov flow and $d\mu$ an invariant smooth probability measure. Then,
$$\mathtt{h}_\mu(X)=\int_M r_u d\mu,$$
where $h_\mu(X)$ is the measure entropy of $\mu$ under the flow.
\end{theorem}

To understand the local implication of the entropy theorem, we need the Birkhoff ergodic theorem for Anosov flows. See \cite{hyp}.

\begin{theorem}[Birkhoff ergodic theorem for Anosov flows] 
Let $X$ be a volume preserving (in particular, contact) Anosov flow with invariant smooth measure $d\mu$. There exist a subset $\mathcal{R}\subset M$ (called the set of regular points) such that $M \backslash \mathcal{R}$ has measure zero and for any $x \in \mathcal{R}$ and $\mu$-integrable function $f:M\rightarrow \mathbb{R}$,
$$\lim_{T\rightarrow \
\infty}\frac{1}{T} \int_0^T f(X^t(x))\ dt=\frac{1}{Vol(\mu)}\int_M f \ d\mu.$$
\end{theorem}

Combining the two theorems implies that for almost every point on $M$ the {\em forward Lyapunov exponent }corresponding to $E^{uu}$, defined using an arbitrary vector $v\in E^{uu}$ as

$$\bar{r}_u(x):=\lim_{T\rightarrow \
\infty}\frac{1}{T} \ln{||X^T_*(v)||}=\lim_{T\rightarrow \
\infty}\frac{1}{T} \int_0^T r_u(X^t(x))\ dt,$$
is equal to the Liouville entropy of $X$.

Therefore, the forward entropy is defined almost everywhere and only depends on the measure entropy of the flow. The same holds for the {\em backward} Lyapunov exponents. Note that in particular, the forward and backward Lyapunov exponents are equal almost everywhere. The same is true for the Lyapunov exponents of the stable invariant bundle $E^{ss}$. This is essentially embedded in the Oseledet's Multiplicative Ergodic Theorem \cite{os} (also see \cite{simic2}).

\begin{corollary}\label{pesincor}
There exists a subset $\mathcal{R}\subset M$ such that $M\backslash \mathcal{R}$ has measure zero, and for any $x\in \mathcal{R}$
$$\bar{r}_u(x)=\lim_{T\rightarrow \
\infty}\frac{1}{T} \int_0^T r_u(X^t(x))\ dt=\int_M r_u \ \frac{d\mu}{Vol(\mu)}=\mathtt{h}_{\frac{\mu}{Vol(\mu)}}=\frac{\mathtt{h}_\mu}{Vol(\mu)},$$
where $\bar{r}_u(x)$ is the Lyapunov exponent of $E^{uu}$ at $x\in M$. A similar statement holds for $\bar{r}_s(x)$, i.e. the Lyapunov exponent of $E^{ss}$.
\end{corollary}


\section{Important examples}\label{s4}

In this section, we would like to review two important classes of critical contact metrics. These are {\em Boothby-Wang fibrations}, corresponding to $S^1$-bundles over surfaces, and {\em examples with $\widetilde{SL}(2;\mathbb{R})$ geometry}, corresponding to the geodesic flows on the unit tangent space of hyperbolic surfaces. One of the main results of this paper, discussed in Section~\ref{s5}, is that in fact these are essentially the only examples of critical contact metrics in dimension 3. Both classes of examples can be generalized to higher dimensions as well. But we only consider them in dimension 3 in this paper.

\subsection{Boothby-Wang fibrations}\label{s41}

Boothby-Wang fibration were introduced in \cite{bw} and their significance in contact metric geometry is thanks to the fact that they characterize vanishing torsion of contact metrics. See Chapter~3 of \cite{bl} for more details.  On an oriented closed surface $\Sigma$, there is a 1-to-1 correspondence between elements of $H^2(\Sigma;\mathbb{Z})$ and $\mathbb{S}^1$-bundles over $\Sigma$. In particular, an element $[\omega]>0\in H^2(\Sigma;\mathbb{Z})$ corresponds to an $\mathbb{S}^1$-bundle $\pi:(M_{(\Sigma,[\omega])},\alpha)\rightarrow \Sigma$, whose fibers are the Reeb flow lines for the connection form $\alpha$ satisfying $d\alpha=\pi^*\omega$, which is a positive contact form in this case.


	Given a Boothby-Wang fibration $(M_{(\Sigma,[\omega])},\alpha)$, it is easy to construct an almost complex structure $\phi$ which is invariant under $X_\alpha$ (i.e. the $\mathbb{S}^1$-action) by starting with an almost complex structure on $(\Sigma,\omega)$ and lifting it along the fibers, i.e. $\mathcal{L}_{X_\alpha}\phi=0$ (see \cite{bl} Example~4.5.4). In particular, $X_\alpha$ is a Killing vector field for $g$.
i.e.
$$\mathcal{L}_{X_\alpha}g=0.$$
Such metric trivially minimizes the Dirichlet energy functional as it satisfies $\mathcal{E}(g)=0$. More interestingly, modifications of these examples give {\em all} contact metrics with $\mathcal{E}(g)=0$. 

Generalizing the above construction to a symplectic orbifold $(\Sigma,\omega)$, we still get a $\mathbb{S}^1$-action on a 3-manifold whose orbits induce a Seifert fibration traced by a Reeb flow. These examples are called {\em generalized Boothby-Wang fibrations} and they are known to be quotients of Boothby-Wang fibrations \cite{thomas}.


Finally, Rukimbira \cite{ruk} shows that any contact metric with $\mathcal{E}(g)=0$ can be approximated by a generalized Boothby-Wang fibration, using the following argument: In this case, $X^t_\alpha$ induces a 1-parameter family of isometries of $g$, whose closure is a compact Abelian Lie subgroup of the space of all isometries, and hence, a (possibly high dimensional) torus. Therefore, the action of $X_\alpha$ can be approximated by an action with only periodic orbits (similar to approximation of a foliation of 2-torus with irrational slope by ones with rational slopes). Rukimbira shows that this action is in fact induced from a nearby Reeb flows and gives explicit examples of such approximations. We gather all these observations in the following theorem.

\begin{theorem}\label{bw}
Let $(M,\alpha)$ be a generalized Boothby-Wang fibration. Then, $(M,\alpha)$ admits a contact metric with $\mathcal{E}(g)=0$. Conversely, if $g$ is a contact metric adapted to $(M,\alpha)$ with $\mathcal{E}(g)=0$. Then, $(M,\alpha)$ can be approximated by generalized Boothby-Wang fibrations.
\end{theorem}

\subsection{Examples with $\widetilde{SL}(2,\mathbb{R})$-geometry}\label{s42}

Geodesic flows of a hyperbolic surfaces $\Sigma$ considered on the unit tangent space $UT\Sigma$ provide a rich class of examples for many areas of mathematics. In particular, they yield the prototypical examples of Anosov systems, classical examples of Reeb flows in contact geometry, and their significance in Riemannian geometry partly relies on the fact that they admit a $\widetilde{SL}(2,\mathbb{R})$ algebra. The theory of contact metric geometry provides a context in which all these areas of mathematics interact, and in fact coincide in the case of critical contact metrics. See \cite{pat} for a thorough reference, \cite{katgeo} for a dynamical approach, \cite{geiges,mcduff} for a contact geometric treatment and \cite{bl,torsion} for appearances in contact metric geometry.

The manifolds with $\widetilde{SL}(2,\mathbb{R})$-geometry are characterized by the existence of a {\em $SL_2$-triple}. That is a basis $(X,Y,Z)$ for the tangent space $TM$, satisfying

$$[Y,Z]=2X, \ \ \ \ [Z,X]=-\lambda Y, \ \ \ \ [X,Y]=\lambda Z,$$
for some constant $\lambda>0$.

We note that the last condition implies that the plane field $\langle X,Y\rangle$ is a contact form and for the 1-form $\alpha$ defined by $\alpha(Y)=\alpha(Z)=0$ and $\alpha(X)=1$, the Reeb vector field is in fact $X_\alpha=X$. Employing the almost complex structure $\phi$ defined by $\phi X=0$, $\phi Y=Z$ and $\phi Z=-Y$, we get a contact metric, which can be explicitly seen to satisfy the Tanno's equation, and hence, is a critical contact metric \cite{torsion}. Such examples are for instance present on the unit tangent space of hyperbolic surfaces, where $X$ generates the geodesic flow of such surface and $Y$ generates the $\mathbb{S}^1$-fibers.

Furthermore, we observe
 that $X$ is in fact an Anosov Reeb flow with $E^{ss}:=\langle Y+Z\rangle$, $E^{uu}:=\langle Y-Z\rangle$ and constant expansion rates $r_u=-r_s=\lambda>0$.
In particular, in the case of surfaces with constant negative curvature $K<0$, this equation can be seen to yield $r_u=-r_s=\sqrt{-K}$.

Anosov contact flows which are virtually smoothly conjugate to these examples are called {\em algebraic} Anosov contact flows, indicating the Lie algebraic context. 
As mentioned in more details in the introduction, algebraic Anosov contact flows are exceptional among Anosov contact flows, as they are unique at satisfying many important rigidity properties. 
Along these lines, Theorem~\ref{cca} of the next section can be interpreted new rigidity result for algebraic Anosov contact flows in terms of contact metric geometry. These are the only Anosov contact flows admitting a critical contact metric. It should be noted that it is well known (see \cite{torsion,blc,abb}) that such examples admit a critical contact metric and we also have the classification result of \cite{abb} for contact metrics on such unit tangent spaces.

\section{Critical contact metrics and the Chern-Hamilton question}\label{s5}

In this section, we give a complete classification of contact manifolds admitting critical contact metrics in dimension 3. This relies on the interplay of critical contact metrics and Anosov dynamics. In the following, we slightly extend the notion of a flow being (projectively) Anosov, by calling a flow (projectively) Anosov on an invariant set $\Lambda\subseteq M$, if $TM|_\Lambda$ admits a splitting as in Definition~\ref{anosov} (Definition~\ref{ca}).

Consider a contact manifold $(M,\alpha)$ with Reeb vector field $X$ and  equipped with a critical metric and let $\Lambda \subseteq M$ be the subset of $M$, on which the associated torsion $\tau$ is non-vanishing (by Proposition~\ref{per}~(a), $\Lambda$ is invariant under the flow of $X$). Now, suppose $(e_1,e_2=\phi e_1)$ is an orthonormal basis for $\ker{\alpha}|_{\Lambda}$ such that $h(e_1)=\lambda e_1$ and $h(e_2)=-\lambda e_2$, where $2\lambda=||\tau||>0$. As discussed in Proposition~\ref{per}, we know $\xi_+:=\langle e_1,X\rangle$ and $\xi_-:=\langle e_2,X\rangle$ are positive and negative contact structures, respectively, transversely intersecting along $X$. Therefore, by \cite{conf,mitsumatsu} $X$ is projectively Anosov on $\Lambda$. Below, we show that $X$ in fact Anosov on $\Lambda$ with expansion and contraction rates equal to $\lambda>0$ and $-\lambda<0$, respectively.


\begin{lemma}\label{cpa}
Using the above notation, $X$ is Anosov on $\Lambda$, on which the plane fields $E^s:=Span(e_1+e_2,X)$ and $E^u:=Span(e_2-e_1,X)$ are invariant under $X$ with expansion rates $r_s=-\lambda$ and $r_u=\lambda$.
\end{lemma}

\begin{proof}
We can compute 
$$[X,e_1+e_2]=[X,e_1]+[X,e_2]=\nabla_X e_1 +\nabla_X e_2 -\nabla_{e_1}X-\nabla_{e_2}X$$
$$=-e_2+e_1+(1+\lambda)e_2-(1-\lambda)e_1=\lambda (e_1+e_2),$$
where we have used Perrone's observations discussed in Proposition~\ref{per}. Similarly, we can show $[X,e_2-e_1]=-\lambda (e_2-e_1)$.
\end{proof}

Our next lemma indicates that the scalar torsion is constant in the presence of a critical contact metric. This means that the invariant set $\Lambda$ in the previous lemma is either empty, or the entire $M$, in which case $X$ is Anosov on the entire $M$.

\begin{lemma}\label{constant}
Let $(M,\alpha,g)$ be critical contact metric. Then, the scalar torsion $||\mathcal{L}_X g||$ is constant. 
\end{lemma}

\begin{proof}
Assume $\lambda$ is not constant. By Sard's theorem, $\lambda:M\rightarrow \mathbb{R}$ has a regular value $c>0$. Let $\Sigma=\lambda^{-1}(c)$ be such smooth compact pre-image. By Proposition~\ref{per}, we have $X \cdot \lambda =0$ and therefore, the surface $\Sigma$ is invariant along the flow generated by $X$. In particular, this means that $\Sigma$ has vanishing Euler number, implying that it is a torus or Klein bottle. This is not possible due to $X$ being Anosov on $\Sigma$ which was shown in Lemma~\ref{cpa}. In particular, this means that $TM|_{\Sigma}$ has an Anosov splitting (as in Definition~\ref{anosov}). But $\Sigma$ being differentiable and invariant implies that it should be everywhere tangent to either $E^u$ or $E^s$ (and therefore a leaf of the stable or unstable weak foliations). In particular, $X|_\Sigma$ is everywhere expanding or everywhere contracting, which is in contradiction with $\Sigma$ being a closed surface.

\end{proof}

Now that for a critical contact metric, the scalar torsion is proved to be constant, the torsion free case $||\mathcal{L}_Xg|| \equiv 0$ is understood by the work of Rukimbira \cite{ruk} and explained in Section~\ref{s4}, while the case $||\mathcal{L}_Xg||  \equiv 2\lambda >0$ requires appealing to the rigidity results in Anosov dynamics. 



\begin{theorem}\label{cca}
Let $(M,\alpha)$ be an Anosov contact manifold admitting a critical contact metric $g$. Then, $(M,\alpha)$ is smoothly strictly contactomorphic to an algebraic Anosov contact metric.
\end{theorem}

\begin{proof}



By Lemma~\ref{constant}, the scalar torsion $||\mathcal{L}_{X} g ||:=2\lambda$ is constant on $M$. We furthermore have $\lambda>0$, since $||\mathcal{L}_{X} g ||\equiv 0$ would imply $g( \nabla_{e_s} X ,e_s )=0$ for $e_s\in E^s$ along any periodic orbit. Therefore, the eigenvalue of return map in $E^s$-direction would be equal to $\pm 1$, which is in contradiction with uniform contraction of $||e_s||$ implied by Anosovity of $X$  (see Proposition~\ref{anosovprop}~(2)). 

Therefore, $X$ is an Anosov contact flow with constant expansion and contraction rates. The following rigidity result in Anosov dynamics then finishes the proof. This was first shown by Green \cite{green} and then, Foulon \cite{entropy} reproved and reinterpreted it as an entropy rigidity phenomena (also see \cite{circle} or \cite{hoz4} for other geometric treatments).

\begin{theorem}(Green 1978 \cite{green}, Foulon 2001 \cite{entropy})
Suppose $X$ generates a contact Anosov flow with constant expansion rates. Then, $X$ is induced from a $SL_2$-triple. More precisely, there are smooth vector fields $e_1,e_2$ such that
$$[e_1,e_2]=2X, \ \ \ \ [e_1,X]=-\lambda e_2, \ \ \ \ [X,e_2]=\lambda e_1,$$
for some constant $\lambda>0$.
\end{theorem}

Note that $E^s:=Span(e_1+e_2,X)$ and $E^u:=Span(e_2-e_1,X)$ are the invariant weak bundles for $X$ and in fact $e_1,e_2$ are the eigenvectors of $h$ for the critical contact metric (see Lemma~\ref{cpa}).

This concludes the proof of Theorem~\ref{cca}, since the existence of such triple means that $X$ generates an algebraic Anosov contact flow (and hence, it is virtually smoothly equivalent to the geodesic flow of a surface of constant negative
curvature. See Section~\ref{s4}). 

\end{proof}




Putting previous results together, we have a complete classification of contact manifolds admitting critical contact metrics, answering the Chern-Hamilton question (Question~\ref{ch}). Recall that by the mentioned work of Deng \cite{deng}, i.e. Theorem~\ref{dength}, these critical contact metrics are in fact the minimizers of the Dirichlet energy. Hence, we have proved the following theorem, which answers Question~\ref{ch}.

\begin{theorem}\label{class}
Let $(M,\alpha)$ be a contact 3-manifold such that the Dirichlet energy functional admits a minimizer. Then,

(i) $(M,\alpha)$ can be approximated by generalized Boothby-Wang fibrations (the case $||\tau||\equiv 0$);

\noindent or

(ii) $(M,\alpha)$ is smoothly strictly contactomorphic to an algebraic Anosov contact manifold (the case $||\tau||\equiv C>0$).

\noindent Conversely, for the contact manifolds listed above the Dirichlet functional admits a minimizer.
\end{theorem}


\section{Anosovity of contact metrics and energy optimization\\ as asymptotic synchronization}\label{s6}

The goal of this section is to show that the Dirichlet energy functional, at least in the Anosov case, has close relation to Reeb dynamics, even in the absence of critical metrics discussed in Section~\ref{s5}. 

We begin by computing the scalar torsion of a (possibly non-critical) contact metric $g$ adapted to a (possibly non-algebraic) Anosov contact manifold $(M,\alpha)$. Given such contact metric, let $e_s \in E^{ss}$ and $e_u \in E^{uu}$ be the unit vectors and $0<\theta<\pi$ denote the angle between the stable and unstable directions. In particular, we have $e_u=\cos{\theta} e_s +\sin{\theta}\phi e_s$, or $\phi e_s=\csc{\theta} e_u -\cot{\theta} e_s$.  Note that $g(e_s,e_u)=\cos{\theta}$ and $d\alpha (e_s,e_u)=-\sin{\theta}$. 

Notice that fixing the norm on $E^{uu}$ and the angle function $\theta$ would determine the norm on $E^{ss}$ as well (the unit vector in $E^{ss}$ direction is the one satisfying $d\alpha(e_s,e_u)=-\sin{\theta}$) and hence the entire contact metric tensor would be determined. In other words, if we fix the norm $||.||_{E^{uu}}$, the entire metric tensor is determined by choosing the line bundle $L\subset \xi$ containing $\phi e_u$, as $\phi e_u$ would be the unique vector in $L$ satisfying $d\alpha(e_u,\phi e_u)=-2$.

This is important in what follows, as we will deform the contact metric through changing the norm on $E^{uu}$ and the angle function $\theta$. The following lemma describes the scalar torsion in terms of these choices, which will be useful in our strategy for optimizing the Dirichlet energy. Not surprisingly, the angle $\theta$ also determines a relation between the expansion rates $r_s$ and $r_u$. More specifically, we have $d\alpha(e_s,e_u)=-\sin{\theta}$ and taking the Lie derivative in the direction of $X$, we get:
$$ r_s(-\sin{\theta})+r_u(-\sin{\theta})=d\alpha([X,e_s],e_u)+d\alpha(e_s,[X,e_u])=-\cos{\theta}(X\cdot \theta),$$
which implies 
$$r_s+r_u=\cot{\theta}(X\cdot \theta).$$

\begin{lemma}\label{sc}
Using the above notation, the scalar torsion can be computed by $$||\mathcal{L}_{X_\alpha} g ||^2 =4r_u^2+[X\cdot \cot{\theta}-2r_u\cot{\theta}]^2.$$
\end{lemma}


\begin{proof}

We have
$$2h(e_s)=[X,\phi e_s]-\phi[X,e_s]=[X,\csc{\theta} e_u -\cot{\theta} e_s]-\phi(-r_s e_s)$$
$$=\big( X\cdot \csc{\theta} -r_u\csc{\theta}+r_s \csc{\theta} \big)e_u+\big( -X\cdot \cot{\theta}\big)e_s.$$


Also note that $\phi e_u=\cos{\theta}(\csc{\theta}e_u-\cot{\theta}e_s)-\sin{\theta}e_s=\cot{\theta}e_u-\csc{\theta}e_s$ and compute
$$2h(e_u)=[X,\phi e_u]-\phi[X,e_u]=[X,\cot{\theta}e_u-\csc{\theta}e_s]-\phi(-r_u e_u)$$
$$\big( -X\cdot \csc{\theta}+r_s \csc{\theta} -r_u\csc{\theta} \big)e_s+\big( X\cdot \cot{\theta}\big)e_u.$$


Now, in $(e_s,e_u)$ basis, we can write
$$2h=
\begin{bmatrix}
 -X\cdot \cot{\theta} \ \ \ \ \ \ \ -X\cdot \csc{\theta} +r_s \csc{\theta} -r_u\csc{\theta} \\
X\cdot \csc{\theta} +r_s \csc{\theta}-r_u\csc{\theta} \ \ \ \ \ \ \ \ \ \ X\cdot \cot{\theta}
 \end{bmatrix}_{(e_s,e_u)}.$$
 
 This means that if we let $\lambda$ be the positive eigenvalue of $h$, we have
 $$4\lambda^2=-Det(2h)=(X\cdot \cot{\theta})^2-(X\cdot \csc{\theta} )^2+(r_u-r_s)^2\csc^2{\theta}$$
 $$=(X\cdot \theta)^2 \big( \csc^4{\theta}-\csc^2{\theta}\cot^2{\theta}\big)+(r_u-r_s)^2\csc^2{\theta}=\csc^2{\theta}\big\{(X\cdot \theta)^2+(r_u-r_s)^2 \big\}$$
 $$=\csc^2{\theta}\big\{(X\cdot \theta)^2+(2r_u-\cot{\theta}(X\cdot \theta))^2 \big\}=\csc^2{\theta}\big\{(X\cdot \theta)^2+\cot^2{\theta}(X\cdot \theta)^2-4r_u\cot{\theta}(X\cdot \theta) +4r_u^2\big\}$$
 $$=\csc^2{\theta}\big\{(\csc^2{\theta}(X\cdot \theta)^2-4r_u\cos{\theta}\csc{\theta}(X\cdot \theta) +4r_u^2\cos^2{\theta}+4r_u^2\sin^2{\theta}\big\}$$
 $$=\csc^2{\theta}\big\{(\csc{\theta}(X\cdot \theta)-2r_u\cos{\theta})^2+4r_u^2\sin^2{\theta}\big\}=4r_u^2+\csc^2{\theta}(\csc{\theta}(X\cdot \theta)-2r_u\cos{\theta})^2$$
$$ =4r_u^2+[X\cdot \cot{\theta}-2r_u\cot{\theta}]^2,$$

finishing the proof of the lemma.
\end{proof}

We use the characterization of the scalar torsion above to prove the main theorem of this section.

\begin{theorem}\label{inf}
Let $(M,\alpha)$ be an Anosov contact 3-manifold. Then,
$$\inf_{g \in \mathcal{M}(\alpha) } \mathcal{E}(g)=\frac{\mathtt{h}^2_{\alpha \wedge d\alpha}(X_\alpha)}{Vol(\alpha \wedge d\alpha)},$$
where $\mathtt{h}_{\alpha \wedge d\alpha}(X_\alpha)$ is the measure entropy of the invariant measure $\alpha \wedge d\alpha$ under the Reeb flow generated by $X_\alpha$. Such infimum is achieved exactly when $(M,\alpha)$ is an algebraic Anosov contact manifold.
\end{theorem}

\begin{proof}
As discussed above, a contact metric in our setting can be described in terms of a norm on $E^{uu}$ and an angle function $\theta:M\rightarrow \mathbb{R}$. We naturally find the infimum of the Dirichlet energy using Lemma~\ref{sc} and in two steps, since
$$\inf_{g \in \mathcal{M}(\alpha) } \mathcal{E}(g)=\inf_{||.||_{E^{uu}}}\inf_{\theta} \mathcal{E}(g)=\inf_{||.||_{E^{uu}}}\inf_{\theta} \int_M \lambda^2 \ \alpha\wedge d \alpha.$$

{\em Step 1.} Lemma~\ref{sc} implies that fixing the norm on $E^{uu}$, we have
$$\inf_{\theta} \mathcal{E}(g)=\inf_{\theta} \int_M \lambda^2 \ \alpha\wedge d \alpha=\int_M r_u^2 \ \alpha\wedge d \alpha.$$
To see this, note that for a fixed $g|_{E^{uu}}$ (and therefore a fixed $r_u$), $\lambda^2$ is minimized when $\theta\equiv \frac{\pi}{2}$, which is equivalent to having $\phi e_u=-e_s$. However, the regularity of such $\phi$ and the resulting contact metric would be as much as the regularity of $E^{ss}$ and $E^{uu}$, i.e. such metric would be only $C^1$ for a general Anosov contact manifold (see the discussion after Definition~\ref{anosov}). To remedy this, such desired angle function can be approximated appropriately by ones induced from smooth choices of almost complex structures (and therefore, smooth choices of the contact metrics). 
 More precisely, if $\epsilon>0$ is arbitrary and $\alpha_u$ is a $C^1$ 1-form with $\ker{\alpha_u}=E^s$, as shown in \cite{hoz3} and reiterated in \cite{massoni}, there is a $C^1$-approximation of $\alpha_u$ with a smooth 1-form $\tilde{\alpha}_u$ such that $||\alpha_u-\tilde{\alpha}_u||_{C^1}<\epsilon$ and $||\mathcal{L}_X\alpha_u-\mathcal{L}_X\tilde{\alpha}_u||_{C^1}<\epsilon$. Now we can choose the smooth line bundle $L:=\xi \cap \ker{\tilde{\alpha}_u}$ to construct an almost complex structure $\tilde{\phi}$ as above (by defining $\tilde{\phi} e_u$ as the unique vector in $L$ satisfying $d\alpha(e_u,\phi e_u)=-2$). This means that if $\tilde{\theta}$ is the angle function corresponding to $\tilde{\phi}$, it can be taken to be uniformly close to $\frac{\pi}{2}$ with $|X\cdot \tilde{\theta}|$ being arbitrary small.


{\em Step 2.} Using Jensen's inequality, we have a lower bound for 
$\mathcal{E}(g)$ as
$$\mathcal{E}(g)=\int_M r_u^2 \ \alpha \wedge d\alpha \geq Vol(\alpha \wedge d\alpha)(\int_M r_u \ \frac{\alpha \wedge d\alpha}{Vol(\alpha \wedge d\alpha)})^2=\frac{\mathtt{h}^2_{\alpha \wedge d\alpha}(X_\alpha)}{Vol(\alpha \wedge d\alpha)}=V\bar{\mathtt{h}}^2$$
where $\bar{\mathtt{h}}:=\frac{\mathtt{h}_{\alpha\wedge d\alpha}}{Vol(\alpha\wedge d\alpha)}$ is the {\em Liouville entropy}, $V:=Vol(\alpha\wedge d\alpha)$ is the contact volume and we have used the Pesin's entropy formula (see Section~\ref{s3}). To show that such minimum is in fact the infimum, we need Birkhoff's ergodic theorem and the {\em asymptotic synchronization process} of \cite{hoz7}.

As noted in Corollary~\ref{pesincor}, applying the Birkhoff's ergodic theorem to Pesin's formula implies that for {\em Lyapunov-regular} points of $M$, i.e. almost everywhere, we have
$$\lim_{T\rightarrow \infty}\frac{1}{T} \int_0^T r_u(X_\alpha^t(x))\ dt=\frac{\mathtt{h}_{\alpha \wedge d\alpha}(X_\alpha)}{Vol(\alpha \wedge d\alpha)}=\bar{\mathtt{h}},$$
implying that for any $x\in \mathcal{R}$, the Lyapunov exponent of $x$ equals the Liouville entropy $\frac{\mathtt{h}_{\alpha \wedge d\alpha}}{Vol(\alpha \wedge d\alpha)}$.

On the other hand, Theorem~4.1 in \cite{hoz7}, implies that given any $\epsilon >0$, the norm on $E^u$ can be defined such that 
$$|r_u-\lim_{T\rightarrow \infty}\frac{1}{T} \int_0^T r_u(X_\alpha^t(x))\ dt|<\epsilon,$$
whenever such limit exists (notice that such limit is independent of the norm defined on $E^u$).

Finally, note that the equality in Jensen's inequality holds, if and only if, the function is constant almost everywhere. Moreover, the error term of the inequality is controlled by the integral of the deviation from being constant (see \cite{jensen}). More precisely, for some Lipschitz constant $K>0$, we have

$$ 0\leq \int_M r_u^2 \ \alpha \wedge d\alpha - \frac{1}{V}(\int_M r_u \ \alpha \wedge d\alpha)^2 \leq K|| r_u - \frac{\int_M r_u \ \alpha \wedge d\alpha}{V} ||_{L^1}<KV\epsilon,$$
where the last inequality follows from the Birkhoff's ergodic theorem. The fact that $\epsilon>0$ can be taken to be arbitrary small finishes the proof for the infimum formula.

On the other hand, if an Anosov contact metric attains such an infimum, it would be a critical contact metric which is Anosov. Therefore, by Theorem~\ref{cca}, the underlying contact manifold is an algebraic Anosov contact manifold in this case.

\end{proof}

\section{Curvature of Anosov contact metrics}\label{s7}

Naturally, the computations of Section~\ref{s6} can be interpreted in terms of the curvature of an Anosov contact metric and we discuss relevant results in this section. More specifically, the convergence of Dirichlet energy to the infimum value in Theorem~\ref{inf} can be refined, using the control that the asymptotic synchronization process (see Theorem~4.1 of \cite{hoz7}) yields over certain quantities related to the second variations of the underlying geometry. 

As pointed out in the introduction, curvature realization problems for contact metrics have been long studied in the literature of contact metric geometry. In particular, the curvature values related to the Reeb vector field $X_\alpha$ are known to be important in the topic. Such results naturally focuses on computing $Ricc(X_\alpha)$, i.e. {\em Ricci-Reeb realization problem}, or the sectional curvature of planes containing $X_\alpha$, also called {\em $\alpha$-sectional curvatures}. As we mentioned in Theorem~\ref{curcm}, we have $Ricci(X_\alpha)=2-2\lambda^2$ for a general contact metric and therefore, we can directly interpret the computations of previous section as Ricci curvature computations. In particular, we have the following global realization obstruction, which is the first of this kind in the literature (see \cite{hoz2} for discussion on global obstructions to Ricci-Reeb realization).

\begin{corollary}\label{ricciglobal}
If $(M,\alpha)$ is an Anosov contact manifold, we have
$$\int_M Ricci(X_\alpha) \ \alpha \wedge d\alpha \leq 2- \frac{2\mathtt{h}^2_{\alpha \wedge d\alpha}(X_\alpha)}{Vol(\alpha \wedge d\alpha)}=2-2V\bar{\mathtt{h}}^2.$$
\end{corollary}

It should be noted that this does not give a local obstruction as for any contact manifold, the condition $Ricci(X_\alpha)=2$ (equivalently, $\mathcal{L}_{X_\alpha} g=0$) can be achieved locally. Naturally, as in Section~\ref{s6}, we can deform the contact metric, in order to make expansions asymptotically uniform and make the value of $Ricci(X_\alpha)$ arbitrary close to this bound almost everywhere. Before stating this result, we note that we can make similar {\em uniformization} of the $\alpha$-sectional curvatures as well. Perrone \cite{torsion} computes such sectional curvatures for a critical contact metric. More specifically, for a critical contact metric, he shows that
$$\kappa(e_1,X_\alpha)=1-\lambda^2+2\lambda\ \ \ \ \ \ \ \ \ \ \ \kappa(e_2,X_\alpha)=1-\lambda^2-2\lambda$$
$$\kappa(e_s=\frac{e_1+e_2}{\sqrt{2}},X_\alpha)=\kappa(e_u=\frac{e_1-e_2}{\sqrt{2}},X_\alpha)=1-\lambda^2,$$
where $e_1$ and $\phi e_2$ are the unit eigenvectors of $h$ as before.

 The asymptotic synchronization process yields similar formulas hold asymptotically for a general Anosov contact metric. However, one should note that $\alpha$-sectional curvatures, unlike $Ricci(X_\alpha)$, also depend on the derivative of expansion rates along the flow lines, which can be taken to be arbitrarily small by Theorem~4.1 in \cite{hoz7}. Moreover, in \cite{hoz2}, it is shown that for any unit vector $e\in \xi$, we can compute the $\alpha$-sectional corresponding to $e$ can be computed as
$$\kappa(e,X_\alpha)=g(\phi e,\nabla_e X_\alpha)^2-g(e,\nabla_e X_\alpha)^2-\frac{\partial}{\partial t} g(e(t),\nabla_{e(t)} X_\alpha)^2\big|_{t=0},$$
where $e(t)=\frac{X^t_*(e)}{||X^t_*(e)||}$. This in particular implies (see Proposition~\ref{cex})
$$\kappa(E^u):=\kappa(e_u,X_\alpha)=g(\phi e,\nabla_e X_\alpha)^2 -r_u^2-X\cdot r_u.$$

Theorem~4.1 of \cite{hoz7} not only implies that for an appropriate contact metric at any point, $r_u$ can be assumed to be close to the Lyapunov exponent $\bar{r}_u$ at that point (or almost everywhere close to the Liouville entropy $\bar{\mathtt{h}}$), but also we can assume $X\cdot r_u$ is arbitrary small. So, the only term we need to understand is $g(\phi e,\nabla_e X_\alpha)$. Note that by computations of the previous section we have
$$g(\phi e,\nabla_e X_\alpha)=g(\phi e,-\phi e_u,-\phi e_u-\phi h(e_u))=-1-g(e_u,h(e_u))$$
$$=-1-g(e_u,\frac{-X\cdot \csc{\theta}+(r_s-r_u) \csc{\theta}}{2}e_s+\frac{X\cdot \cot{\theta}}{2}e_u)$$
$$=-1+\frac{X\cdot \csc{\theta}+(r_u-r_s) \csc{\theta}}{2}\cos{\theta}-\frac{X\cdot \cot{\theta}}{2}$$
and therefore, similar to Theorem~\ref{inf}, by only taking the angle function $\theta$ arbitrary close to $\frac{\pi}{2}$ with $X\cdot \theta$ arbitrary small, we can assume $g(\phi e,\nabla_e X_\alpha)$ is arbitrary close to $-1$. This means that at any point $\kappa(E^u)$ can be taken to be arbitrary close to $1-\bar{r}_u^2$ (assuming such exponent exists at the point). Similar computations hold for $\kappa(E^s)$ and using the convergence in the Birkhoff ergodic theorem as in Section~\ref{s6}, we obtain the following {\em asymptotic} curvature realization. In fact, for any plane field containing $X_\alpha$, the sectional curvature converges almost everywhere to the one determined under the assumption of critical metric by Perrone \cite{torsion}. Naturally, the global convergence is in general only as good as the one given the ergodic theorem, i.e. in {\em Lebegue measure}. Note that the convergence on periodic orbits is uniform, thanks to their compactness and the fact the Lyapunov exponents for periodic points are well defined and determined by the eigenvalues of their return maps (see Proposition~\ref{anosovprop}).

\begin{theorem}\label{pinch}(Asymptotic curvature pinching for Anosov contact metrics)
Suppose $(M,\alpha)$ is an Anosov contact 3-manifold, $\bar{\mathtt{h}}:=\frac{\mathtt{h}_{\alpha \wedge d\alpha}(X_\alpha)}{Vol(\alpha \wedge d\alpha)}$ is the Liouville entropy of such flow, $V=Vol(\alpha \wedge d\alpha)$ and $\epsilon>0$. There exists a sequence of contact metrics $\{g_i \}_{i\in \mathbb{N}}$ adapted to $(M,\alpha)$, such that their Ricci and sectional curvature operators $Ricci_i$ and $\kappa_i$ satisfy the following:

(1) the sequence of smooth functions $\{Ricci_i(X_\alpha)\}_{i\in\mathbb{N}}$ converge in Lebegue measure to the constant function $R\equiv 2-2\bar{\mathtt{h}}^2$,

(2) both sequences of smooth functions $\{ \kappa_i(E^s)\}_{i\in\mathbb{N}}$ and $\{ \kappa_i(E^u)\}_{i\in\mathbb{N}}$ converge in Lebegue measure to the constant function $\kappa \equiv 1-\bar{\mathtt{h}}^2$, where $\kappa_i(E^u)$ and $\kappa_i(E^s)$ are the sectional curvature functions corresponding to the invariant bundles $E^u$ and $E^s$, respectively,

(3) if $\gamma$ is a periodic orbit of $X_\alpha$ with period $T$ and the eigenvalues of its return map corresponding to $E^u$ and $E^s$ being $\lambda_u$ and $\lambda_s$, respectively, then $\{Ricci_i(X_\alpha)\}_{i\in\mathbb{N}}|_\gamma$ converges uniformly to $2-2(\frac{\ln{|\lambda_u}|}{T})^2=2-2(\frac{\ln{|\lambda_s|}}{T})^2$ and similarly, both $\kappa_i(E^u)|_\gamma$ and $\kappa_i(E^s)|_\gamma$ converge uniformly to $1-(\frac{\ln{|\lambda_u}|}{T})^2=1-(\frac{\ln{|\lambda_s|}}{T})^2$.
\end{theorem}

In the case of algebraic Anosov contact manifolds, we can avoid appealing to ergodic theorem, since the invariant bundles are smooth and for some contact metric we have $r_u=-r_s=\frac{\mathtt{h}_{\alpha\wedge d\alpha}}{Vol(\alpha \wedge d\alpha)}=\bar{\mathtt{h}}$. Therefore, we can characterize all function which can be realized as $Ricci(X_\alpha)$ in terms of the Liouville entropy $\bar{\mathtt{h}}$. In particular, when $(M=UT\Sigma,\alpha)$ is the canonical contact manifold corresponding to the geodesic flow of a surface $\Sigma$ of constant negative curvature $K<0$, we have $\bar{\mathtt{h}}=\sqrt{-K}$ (see Section~\ref{s4}). Letting $\sigma:=\frac{\cot{\theta}}{2}$ and $\eta$ being the unique, up to a constant summation, function satisfying $r_u=\bar{\mathtt{h}}+X\cdot \eta$ (see Propostion~\ref{anosovprop}~(3)) we achieve the following solution to the Ricci-Reeb realization problem for algebraic Anosov contact manifolds.

\begin{theorem}\label{alg}(Ricci-Reeb realization formula for algebraic Anosov contact manifolds)
Let $(M,\alpha)$ be an algebraic Anosov contact manifold with Liouville entropy $\bar{\mathtt{h}}=\frac{\mathtt{h}_{\alpha \wedge d\alpha}}{Vol(\alpha\wedge d\alpha)}$. Then, for a smooth real function $f:M\rightarrow \mathbb{R}$, the followings are equivalent:

(1) For some adapted contact metric, we have $Ricci(X_\alpha)=f$ everywhere.

(2) for real functions $\eta,\sigma:M\rightarrow \mathbb{R}$ we have
$$f=2-2(\bar{\mathtt{h}}+X_\alpha\cdot \eta)^2-2[X_\alpha\cdot \sigma-2\sigma(\bar{\mathtt{h}}+X_\alpha\cdot \eta) ]^2.$$
\noindent In particular, if $(UT\Sigma,\alpha)$ is the canonical contact manifold corresponding to a surface of constant curvature $K<0$, a function $f$ can be realized as $Ricci(X_\alpha)$, if and only if, it can be written as
$$f=2-2(\sqrt{-K}+X_\alpha\cdot \eta)^2-2[X_\alpha\cdot \sigma-2\sigma(\sqrt{-K}+X_\alpha\cdot \eta) ]^2,$$
for some functions $\eta,\sigma:M\rightarrow \mathbb{R}$.
\end{theorem}





\Addresses

\begin{thebibliography}{00}

\bibitem{abbreeb} Abbondandolo, A., Alves, M. R., Sağlam, M., \& Schlenk, F. (2023). Entropy collapse versus entropy rigidity for Reeb and Finsler flows. Selecta Mathematica, 29(5), 67.

\bibitem{abb} Abbassi, Mohamed TK, and Giovanni Calvaruso. {\em g-Natural contact metrics on unit tangent sphere bundles.} Monatshefte für Mathematik 151 (2007): 89-109.

\bibitem{yam2} Aubin, Thierry. {\em Équations différentielles non linéaires et probleme de Yamabe concernant la courbure scalaire.} J. Math. Pures Appl.(9) 55 (1976): 269-296.




\bibitem{blc} Blair, David E. {\em Critical associated metrics on contact manifolds.} Journal of the Australian Mathematical Society 37.1 (1984): 82-88.

\bibitem{bl} Blair, David E. {\em Riemannian geometry of contact and symplectic manifolds.} Springer Science \& Business Media, 2010.


\bibitem{blperr} Blair, David E., and D. Peronne. {\em Conformally Anosov flows in contact metric geometry.} Balkan Journal of Geometry and Its Applications 3.2 (1998): 33-46.

\bibitem{bw} Boothby, William M., and Hsieu-Chung Wang. {\em On contact manifolds.} Annals of Mathematics (1958): 721-734.


\bibitem{ch} Chern, Shiing-Shen, and Richard S. Hamilton. {\em On Riemannian metrics adapted to three-dimensional contact manifolds.} Arbeitstagung Bonn 1984: Proceedings of the meeting held by the Max-Planck-Institut für Mathematik, Bonn June 15–22, 1984. Springer Berlin Heidelberg, 1985.

\bibitem{jensen} Costarelli, Danilo, and Renato Spigler. {\em How sharp is the Jensen inequality?.} Journal of Inequalities and Applications 2015 (2015): 1-10.

\bibitem{deng} Deng, Shangrong. {\em The second variation of the Dirichlet energy on contact manifolds.} Kodai Mathematical Journal 14.3 (1991): 470-476.

\bibitem{conf} Eliashberg, Yakov, and William P. Thurston. {\em Confoliations.} Vol. 13. American Mathematical Soc., 1998.

\bibitem{etnyre} Etnyre, John B., Rafal Komendarczyk, and Patrick Massot. {\em Tightness in contact metric 3-manifolds.} Inventiones mathematicae 188.3 (2012): 621-657.

\bibitem{etnyre2} Etnyre, John, Rafal Komendarczyk, and Patrick Massot. {\em Quantitative Darboux theorems in contact geometry.} Transactions of the American Mathematical Society 368.11 (2016): 7845-7881.

\bibitem{hyp} Fisher, Todd, and Boris Hasselblatt. {\em Hyperbolic flows.} 2019.

\bibitem{entropy} Foulon, Patrick. {\em Entropy rigidity of Anosov flows in dimension three.} Ergodic Theory and Dynamical Systems 21.4 (2001): 1101-1112.

\bibitem{foulonh} Foulon, Patrick, and Boris Hasselblatt. {\em Anosov contact flows on hyperbolic 3–manifolds.} Geometry \& Topology 17.2 (2013): 1225-1252.

\bibitem{geiges} Geiges, Hansjörg. {\em An introduction to contact topology.} Vol. 109. Cambridge University Press, 2008.

\bibitem{circle} Geiges, Hansjörg, and Jesús Gonzalo. {\em Contact geometry and complex surfaces.} Inventiones mathematicae 121 (1995): 147-209.

\bibitem{ghys2} Ghys, Étienne. {\em Flots d'Anosov sur les 3-variétés fibrées en cercles.} Ergodic Theory and Dynamical Systems 4.1 (1984): 67-80.

\bibitem{ghys} Ghys, Étienne. {\em Rigidité différentiable des groupes fuchsiens.} Publications Mathématiques de l'IHÉS 78 (1993): 163-185.


\bibitem{green} Green, L. W. {\em Remarks on uniformly expanding horocycle parameterizations.} Journal of Differential Geometry 13.2 (1978): 263-271.





\bibitem{hps} Hirsch, Morris W., Charles Chapman Pugh, and Michael Shub. {\em Invariant manifolds.} Bulletin of the American Mathematical Society 76.5 (1970): 1015-1019.

\bibitem{hoz1} Hozoori, Surena. {\em Dynamics and topology of conformally Anosov contact 3-manifolds.} Differential Geometry and its Applications 73 (2020): 101679.

\bibitem{hoz4} Hozoori, Surena. {\em On Anosovity, divergence and bi-contact surgery.} Ergodic Theory and Dynamical Systems 43.10 (2023): 3288-3310.

\bibitem{hoz2} Hozoori, Surena. {\em Ricci curvature, Reeb flows and contact 3-manifolds.} The Journal of Geometric Analysis (2021): 1-26.

\bibitem{hoz7} Hozoori, Surena. {\em Strongly adapted contact geometry of Anosov 3-flows.} Journal of Fixed Point Theory and Applications 27.2 (2025): 37.

\bibitem{hoz3} Hozoori, Surena. {\em Symplectic geometry of Anosov flows in dimension 3 and bi-contact topology.} Advances in Mathematics 450 (2024): 109764.



\bibitem{katgeo} Katok, Anatole. {\em Entropy and closed geodesies.} Ergodic theory and dynamical Systems 2.3-4 (1982): 339-365.


\bibitem{lisca} Lisca, Paolo, and Gordana Matic. {\em Transverse contact structures on Seifert fibered 3-manifolds,} Algebr. Geom. Topol. 4 (2004) 1125–1144.

\bibitem{massoni} Massoni, Thomas. {\em Anosov flows and Liouville pairs in dimension three.} Algebraic \& Geometric Topology 25.3 (2025): 1793-1838.

\bibitem{mcduff} McDuff, Dusa. {\em Symplectic manifolds with contact type boundaries.} Inventiones mathematicae 103.1 (1991): 651-671.

\bibitem{mitsumatsu} Mitsumatsu, Yoshihiko. {\em Anosov flows and non-Stein symplectic manifolds.} Annales de l'institut Fourier. Vol. 45. No. 5. 1995.

\bibitem{mps} Mitsumatsu, Y., Peralta-Salas, D. and Slobodeanu, R. (2025), {\em On the existence of critical compatible metrics on contact 3-manifolds.} Bull. London Math. Soc., 57: 79-95. https://doi.org/10.1112/blms.13183

\bibitem{ol} Olszak, Zbigniew. {\em On contact metric manifolds.} Tohoku Mathematical Journal, Second Series 31.2 (1979): 247-253.

\bibitem{os} Oseledec, Valery Iustinovich. {\em A multiplicative ergodic theorem, Lyapunov characteristic numbers for dynamical systems.} Transactions of the Moscow Mathematical Society 19 (1968): 197-231.

\bibitem{pat} Paternain, Gabriel P. {\em Geodesic flows.} Vol. 180. Springer Science \& Business Media, 2012.

\bibitem{radu} Peralta-Salas, Daniel, and Radu Slobodeanu. {\em Contact structures and Beltrami fields on the torus and the sphere.} arXiv preprint arXiv:2004.10185 (2020).

\bibitem{hom} Perrone, Domenico. {\em Homogeneous contact Riemannian three-manifolds.} Illinois Journal of Mathematics 42.2 (1998): 243-256.

\bibitem{taut} Perrone, Domenico. {\em Taut contact hyperbolas on three-manifolds.} Annals of Global Analysis and Geometry 60.3 (2021): 735-765.

\bibitem{torsion} Perrone, Domenico. {\em Torsion and conformally Anosov flows in contact Riemannian geometry.} Journal of Geometry 83 (2005): 164-174.

\bibitem{thomas} Thomas, C. B. {\em Almost regular contact manifolds.} Journal of Differential Geometry 11.4 (1976): 521-533.


\bibitem{ruk} Rukimbira, Phillippe. {\em Chern-Hamilton Conjecture and K-contactness.} Hous. Jour. of Math. 21, No 4, (1995), 709-718.

\bibitem{yam1} Schoen, Richard. {\em Conformal deformation of a Riemannian metric to constant scalar curvature.} Journal of Differential Geometry 20.2 (1984): 479-495.

\bibitem{simic} Simić, Slobodan. {\em Codimension one Anosov flows and a conjecture of Verjovsky.} Ergodic Theory and Dynamical Systems 17.5 (1997): 1211-1231.

\bibitem{simic2} Simić, Slobodan N. {\em Oseledets regularity functions for Anosov flows.} Communications in mathematical physics 305 (2011): 1-21.

\bibitem{tanno} Tanno, Shukichi. {\em Variational problems on contact Riemannian manifolds.} Transactions of the American Mathematical society 314.1 (1989): 349-379.

\bibitem{yam} Yamabe, Hidehiko. {\em On a deformation of Riemannian structures on compact manifolds.} (1960): 21-37.

\end{thebibliography}
\end{document}